\documentclass{amsart}
\usepackage{amssymb,latexsym}
\usepackage[colorlinks]{hyperref}

\theoremstyle{plain}
\newtheorem{theorem}{Theorem}[section]

\newtheorem{lemma}[theorem]{Lemma}
\newtheorem{corollary}[theorem]{Corollary}

\theoremstyle{definition}
\newtheorem{definition}[theorem]{Definition}

\theoremstyle{remark}
\newtheorem{remark}[theorem]{Remark}

\DeclareMathOperator{\sgn}{sgn}
\DeclareMathOperator{\ord}{ord}

\begin{document}
\title[Straightening Law]{On the Straightening Law for Minors of a Matrix }
\author{Richard G. Swan}
\address{Department of Mathematics\\
The University of Chicago\\
Chicago, IL 60637}
\email{swan@math.uchicago.edu}
\dedicatory{In memory of Gian Carlo Rota}
\footnote{I would like to thank Darij Grinberg for many
corrections to an earlier version of this paper}

\begin{abstract}
We give a simple new proof for the straightening law
of Doubilet, Rota, and Stein using a generalization
of the Laplace expansion of a determinant.
\end{abstract}

\maketitle

\section{Introduction}
The straightening law was proved in \cite{DRS} by
Doubilet, Rota, and Stein generalizing earlier work
of Hodge \cite{H}. Since then a number of other
proofs have been given \cite{DKR,DEPy,ABW}. The object
of the present paper is to offer yet another proof of
this result based on a generalization of the Laplace
expansion of a determinant. This proof has the advantage
(to some!) of not requiring any significant amount of
combinatorics, Young diagrams, etc. On the other hand,
for the same reason, it does not show the interesting
relations between the straightening law and invariant
theory but these are very well covered in the above
references and in \cite{DEPh}. For completeness, I have
also included a proof of the linear independence of the
standard monomials.

\section{Laplace Products}
Let $X=(x_{ij})$ be an $m\times n$ matrix where $1\le i\le m$
and $1\le j\le n$. If $A\subseteq\{1,\ldots,m\}$,
$B\subseteq\{1,\ldots,n\}$, and $A$ and $B$
have the same number of elements
we define $X(A|B)$ to be the
minor determinant of $X$ with row indices in $A$
and column indices in $B$. I will usually just write
$(A|B)$ for $X(A|B)$ when it is clear what $X$ is. I will
write $|A|$ for the number of elements in $A$. We set
$X(A|B)=0$ if $|A|\neq |B|$.

I will write $\widetilde A$ for the complement $\{1,\ldots,m\}-A$
and $\widetilde B$ for $\{1,\ldots,n\}-B$. Also $\sum A$ will
denote the sum of the elements of $A$.

\begin{definition}\label{def1}
If $m=n$, we define the Laplace product $X\{A|B\}$ to be
$X\{A|B\}=(-1)^{\sum A+\sum B}(A|B)(\widetilde A|\widetilde B)$.
\end{definition}

If $X$ is understood, I will just write $\{A|B\}$ for $X\{A|B\}$.
This notation is, of course, for this paper only and is
not recommended for general use.

The terminology comes from the Laplace expansion 
\begin{equation}\label{eq1}
\det X = \sum_{|S|=|B|}\{S|B\} = \sum_{|T|=|A|}\{A|T\}
\end{equation}
where $A$ and $B$ are fixed.

The following lemma explains the sign in Definition~\ref{def1}.

\begin{lemma}\label{lem1}
Let $y_{ij}=x_{ij}$ if $(i,j)$ lies in $A\times B$ or in
$\widetilde A\times\widetilde B$ and let $y_{ij}=0$
otherwise. Then $X\{A|B\}=\det (y_{ij})$.
\end{lemma}

\begin{proof}
Rearrange the rows and columns of $Y=(y_{ij})$ so that those with
indices in $A$ and $B$ lie in the upper left hand corner. The
resulting matrix has determinant $(A|B)(\widetilde A|\widetilde B)$.
The sign of the permutation of rows and columns is
$(-1)^{\sum A+\sum B}$ by the next lemma.
\end{proof}

\begin{lemma}
Let $A=\{a_1<\cdots<a_p\}$ be a subset of $\{1,\ldots,n\}$ and
let $\{c_1<\cdots<c_q\}=\{1,\ldots,n\}-A$. Then the sign of
the permutation taking $\{1,\ldots,n\}$ to
$\{a_1,\ldots,a_p,c_1,\ldots,c_q\}$
is $(-1)^{\sum (a_i-i)}$.
\end{lemma}

\begin{proof}
Starting with $\{1,\ldots,n\}$ 
move $a_1$ to position $1$, then $a_2$ to position $2$, etc., each
time keeping the remaining elements in their given order. The number of
transpositions used is $\sum (a_i-i)$.
\end{proof}

The Laplace expansion (\ref{eq1}) gives us a non-trivial
relation between the Laplace products of $X$. This suggests
looking for more general relations of the form
\begin{equation}\label{eq2}
\sum_i a_i\{S_i|T_i\}=0
\end{equation}
with constant $a_i$.
          
Since $\{A|B\}$ is multilinear in the rows of $X$ it will suffice
to check a relation (\ref{eq2}) for the case in which the rows
of $X$ are all of the form $0,\ldots,0,1,0,\ldots,0$. Since $\{A|B\}$
is also multilinear in the columns of $X$, all terms of (\ref{eq2})
will be $0$ unless there is a $1$ in each column. Therefore it will
suffice to check a relation (\ref{eq2}) for the case where $X$ is
a permutation matrix, $X=P(\sigma^{-1})=(\delta_{\sigma i,j})$. To
do this we first compute $\{A|B\}$ for this $X$.

\begin{lemma}
If $X=(\delta_{\sigma i,j})$ is a permutation matrix
$P(\sigma^{-1})$ then $\{A|B\}=\sgn\sigma$ if $\sigma A=B$
and $\{A|B\}=0$ otherwise.
\end{lemma}

This is immediate from Lemma~\ref{lem1} and the fact that $\det X$
becomes $0$ if any entry $1$ is replaced by $0$.

\begin{corollary}
A relation $\sum_i a_i\{S_i|T_i\}=0$ between Laplace
products (with constant $a_i$) holds if and only if for each $\sigma$
in $\mathcal{S}_n$ we have $\sum a_i=0$ over those $i$ with $\sigma S_i=T_i$.
\end{corollary}

\begin{theorem}\label{th1}
For given $A$ and $B$ we have
\begin{displaymath}
\sum_{V\subseteq B}\{A|V\}=\sum_{U\supseteq A}\{U|B\}.
\end{displaymath}
\end{theorem}

\begin{proof}
As above, it will suffice to check this when $X$ is 
a permutation matrix, $X=P(\sigma^{-1})=(\delta_{\sigma i,j})$.
For this $X$ the right hand side is $\sgn\sigma$ if $\sigma A\subseteq B$
and otherwise is $0$. The same is true of the left hand side.
\end{proof}

In particular, we recover the Laplace expansion (\ref{eq1}) by setting
$A=\emptyset$ or by setting $B=\{1,\ldots,n\}$.

We will show later that, for generic $X$, all linear relations between Laplace
products  are consequences of those in Theorem~\ref{th1}.

\begin{corollary}\label{cor1}
For given $A$, $B$, and for $C\subseteq B$ we have
\begin{displaymath}
\sum_{C\subseteq V\subseteq B}\{A|V\}=
\sum_{\substack{U\supseteq A\\W\subseteq C}}
(-1)^{|W|}\{U|B-W\}.
\end{displaymath}
\end{corollary}

\begin{proof}
By Theorem~\ref{th1}, the right hand side is
\begin{displaymath}
\sum_{W\subseteq C}(-1)^{|W|}\sum_{V\subseteq B-W}\{A|V\}=
\sum_{V\subseteq B}\thickspace\sum_{W\subseteq C-V}(-1)^{|W|}\{A|V\}
\end{displaymath}
which is equal to the left hand side since the inner sum is $0$ unless
$C\subseteq V$.
\end{proof}

Recall that $\widetilde A$ here denotes the complement of $A$.

\begin{corollary}\label{cor2}
For given $A$, $B$,
\begin{displaymath}
\sum_{\substack{U\supseteq A\\W\supseteq B}}
(-1)^{|\widetilde W|}\{U|W\}=
\sum_{V\subseteq B}\{A|\widetilde V\}.
\end{displaymath}
\end{corollary}

\begin{proof}
Set $B=\{1,\ldots,n\}$ in Corollary~\ref{cor1}, getting
\begin{displaymath}
\sum_{V\supseteq C}\{A|V\}=
\sum_{\substack{U\supseteq A\\W\subseteq C}}
(-1)^{|W|}\{U|\widetilde W\}.
\end{displaymath}
Replace $C$, $V$, $W$ by $\widetilde B$, $\widetilde V$, $\widetilde W$
where $B$ is now the $B$ given in Corollary~\ref{cor2}.
\end{proof}

\section{Straightening Laplace products}
If $S=\{s_1<\cdots<s_p\}$ and $T=\{t_1<\cdots<t_q\}$ are subsets of
$\{1,\cdots,n\}$, we define a partial ordering $S\leq T$ as in \cite{DEPh}
to mean $p\geq q$ and $s_\nu\leq t_\nu$ for all $\nu\leq q$. Equivalently
$S\leq T$ if and only if $|S\cap\{1,\cdots,r\}|\geq|T\cap\{1,\cdots,r\}|$
for all $1\leq r\leq n$. Note that $S\supseteq T$ implies $S\leq T$
and $S\supset T$ implies $S < T$.

As above let $\widetilde S$ be the complement $\{1,\ldots,n\}-S$. For want
of a better terminology, I will say that  $S$ is good if $S\leq\widetilde{S}$
and that $S$ is bad otherwise.

The following theorem is our straightening law for Laplace products.

\begin{theorem}\label{th2}
For any $\{A|B\}$ we have $\{A|B\}=\sum\pm\{A_i|B_i\}$ where $A_i\leq A$,
$B_i\leq B$ and $A_i$ and $B_i$ are good.
\end{theorem}
Note that some of the $(A_i,B_i)$ may be equal.

\begin{proof}
By induction on the set of pairs of subsets of $\{1,\cdots,n\}$
partially ordered by $(A,B)\leq(C,D)$ if $A\leq C$ and $B\leq D$, it is
sufficient to prove that if $A$ is bad, then $\{A|B\}=\sum\pm \{A_i|B_i\}$
with $A_i<A$ and $B_i\leq B$, and, similarly, if $B$ is bad then
$\{A|B\}=\sum\pm \{A_i|B_i\}$ with $A_i\leq A$ and $B_i<B$.
Each statement implies the other by transposing our matrix.

Suppose first that $|A|=|B|<n/2$. Note both $A$ and $B$ are bad in this
case. By Corollary~\ref{cor2} we have
\begin{displaymath}
\sum_{\substack{U\supseteq A\\W\supseteq B}}
\pm\{U|W\}=0.
\end{displaymath}
since the other side in Corollary~\ref{cor2} is $0$ because $|A|<n/2<|\widetilde{V}|$. One term of
this sum is $\pm\{A|B\}$ while all other nonzero terms have the form $\pm\{U|W\}$
where $U<A$ and $W<B$.

In the remaining case $|A|=|B|\geq n/2$ we use an argument similar to
that of Hodge \cite{H}. Suppose $B$ is bad. Let $B=\{i_1<\cdots<i_p\}$
and let $\widetilde B=\{j_1<\cdots<j_q\}$. Since $B$ is bad and $q\leq p$,
we have $i_\nu>j_\nu$ for some $\nu$ which we choose minimal. Let
$D=B\cup\{j_1,\ldots,j_\nu\}$ and let
$C=\{j_1<\cdots<j_\nu<i_\nu<\cdots<i_p\}$. Apply Corollary~\ref{cor1} with $D$
in place of $B$. The left hand side is $0$ since $|A|=p<|C|=p+1$ so we get
\begin{displaymath}
\sum_{\substack{U\supseteq A\\W\subseteq C}}
(-1)^{|W|}\{U|D-W\}=0.
\end{displaymath}
The term with $U=A$ and $W=\{j_1,\ldots,j_\nu\}$ is $\pm\{A|B\}$. This is the
only term of the form $\pm\{U|B\}$ since $|U|=|B|$ and $U\supseteq A$. In the
remaining terms we have $U\leq A$ since $U\supseteq A$. In these terms
$D-W\neq B$ so if $W\subseteq \{j_1<\cdots<j_\nu\}$ then
$D-W\supset B$ and therefore $D-W<B$.
In any case we have $|D-W| = |U| \geq |A| = |B| $.
 If $W$ contains some $i_\mu$ then
$D-W$ is obtained from $B$ by removing some (and at least one) of
the elements $\{i_\nu<\cdots<i_p\}$ and replacing them by at least as many
of the smaller elements $\{j_1<\cdots<j_\nu\}$. This operation does not
decrease the size of the sets $(D-W)\cap\{1,\cdots,r\}$ for $r\geq 1$
so $|(D-W)\cap\{1,\cdots,r\}|\geq|B\cap\{1,\cdots,r\}|$.  Therefore $D-W<B$.
\end{proof}

\section{The straightening law for minors}
We now use a simple trick to generalize Theorem~\ref{th2} to the case of
products of any two minors of a rectangular matrix $X=(x_{ij})$ where
$1\le i\le m$ and $1\le j\le n$. Recall that $(S|T)$ is the
minor of $X$ with row indices in $S$ and column indices in $T$. We
set $(S|T)=0$ if $|S|\neq|T|$. As above we write
$(S',S'')\leq (T',T'')$ if $S'\leq T'$ and $S''\leq T''$.

The following theorem is the straightening law for minors.

\begin{theorem}\label{th3}
If $(S',T')\nleq(S'',T'')$ then
\begin{displaymath}
(S'|T')(S''|T'')=\sum\pm(S'_i|T'_i)(S''_i|T''_i)
\end{displaymath}
where $(S'_i,T'_i)<(S',T')$ and $(S'_i,T'_i)\leq(S''_i,T''_i)$.
\end{theorem}

We need two lemmas for the proof. By an order preserving map I mean one 
satisfying $f(a) \leq f(b) $ if $a \leq b$.

\begin{lemma}\label{lem2}
Let $U'$ and $U''$ be finite subsets of a totally ordered set and let $k = |U'| + |U''|$. Then
there is an order preserving map $f:K = \{1,\dots,k\} \to U'\cup U''$
and disjoint subsets $K'$ and $K''$ of $K$ with $K = K' \sqcup K''$ such that 
\begin{description}
\item[(a)]{$f$ maps $K'$ isomorphically onto $U'$ and  $K''$ isomorphically onto $U''$}
\item[(b)]{If $a,b\in K$, $f(a) = f(b)$, and $a < b$, then $a \in K'$ and $b \in K''$.}
\end{description}
\end{lemma}

\begin{proof}
Let $L'$ and $L''$ be disjoint sets in $1-1$ correspondence with
$U'$ and $U''$. Let $L = L'\sqcup L''$ and let
$f:L\rightarrow U'\cup U''$ map $L'$ bijectively onto
$U'$ and $L''$ bijectively onto $U''$.
Define an ordering on $L$ by setting $a < b$ if $f(a) < f(b)$
or if $f(a)  =  f(b)$ with $a\in L'$ and $b\in L''$.
It is easy to check that this defines a total ordering preserved by $f$.
Since $L$ is a totally ordered finite set of order $k$ it is isomorphic to
$K = \{1,\dots,k\}$ and we can substitute  $K$ for $L$ letting $K'$ and $K''$
correspond to $L'$ and $L''$. (a) is clear and (b) follows from the definition of the ordering/
\end{proof}

If $f$ is order preserving and is injective on $P$ and $Q$ then
$P \leq Q$ implies $f(P) \leq f(Q)$ since the  injectivity guarantees that
$|f(P)| = |P|$ and similarly for $Q$.

\begin{lemma}\label{lem3}
In the situation of Lemma ~\ref{lem2} if $P$ is a subset of $K$ on which $f$
is injective then $P <  K'$ implies that $f(P) < f(K') = U'$.
\end{lemma}

\begin{proof}
It is clear that $f(P)\leq f(K')$. We must show that $f(P)\neq f(K')$.
Suppose $f(P) = f(K')$. Injectivity shows that $|P| = |f(P)| = |f(K')| = |K'|$.
Let $P  = \{u_1<\dots< u_p\}$ and $K' = \{v_1<\dots< v_p\}$. Since $f$
is order preserving and injective on $P$ and on $K$ we have 
$f(P)  = \{f(u_1)<\dots<f(u_p)\}$ and $f(K') = \{f(v_1)<\dots<f( v_p)\}$
Now $u_\nu \leq v_\nu$ for all $\nu$ since $P < K'$ but
$f(u_\nu) =  f(v_\nu)$  for all $\nu$ since
$f(P) = f(K')$. Since $P < K'$, $u_\nu = v_\nu$ can't hold for all $\nu$
otherwise $P$ and $K'$ would be equal, so for some $\nu$ we have
$u_\nu <  v_\nu$. But since $f(u_\nu) =  f(v_\nu)$,  Lemma~\ref{lem2}(b)
shows that $v_\nu$ must lie in $K''$ which is a contradiction.
\end{proof}

\begin{proof}[Proof of Theorem~\ref{th3}]
We can assume that $|S'| = |T'|$ and $|S''| = |T''|$
since otherwise the left hand side is $0$.
Let $k = |S'| + |S''| = |T'| + |T''|$ and apply Lemma~\ref{lem2} to
$U' = S'$ and $U'' = S''$ 
getting an order preserving map
$\varphi:I\rightarrow S'\cup S''$ with
$I = I'\sqcup I''$(disjoint union), $I'$ mapping isomorphically
to $S'$, and $I''$ mapping isomorphically to $S''$. Similarly define
$\psi:J\rightarrow T'\cup T''$ with $J =J'\sqcup J''$.
Note that $I = J = \{1,\dots,k\}$.
We call them $I$ and $J$ to distinguish their use as row and column indices.

Define a $k\times k$ matrix $Y$ indexed by $I$ and $J$ by
setting $y_{ij}=x_{\varphi(i)\psi(j)}$. Then, for $P\subseteq I$
and $Q\subseteq J$, we have $Y(P|Q)=X(\varphi(P)|\psi(Q))$ if
$\varphi|P$ and $\psi|Q$ are injective while $Y(P|Q)=0$ otherwise
since two rows or columns will be equal.

By Theorem~\ref{th2} we have $Y\{I'|J'\}=\sum\pm Y\{I'_i|J'_i\}$ which we can write as
\begin{equation}\label{eq3}
Y(I'|J')Y(I''|J'')=\sum\pm Y(I'_i|J'_i)Y(	I''_i|J''_i)
\end{equation}
where $I''_i=I-I'_i=\widetilde I'_i$ and $J''_i=J-J'_i=\widetilde J'_i$.
By Theorem~\ref{th2} we see that $I'_i\leq I'$, $J'_i\leq J'$
and that $I'_i$ and $J'_i$ are good
so that $I'_i\leq I''_i$ and $J'_i\leq J''_i$.
By omitting all $0$ terms in (\ref{eq3}) we can insure that $\varphi$ is injective on
all $I'_i$ and all $I''_i$ and that $\psi$ is injective on all $J'_i$ and  $J''_i$.

Let $S'_i=\varphi(I'_i)$, $S''_i=\varphi(I''_i)$ and similarly
for $T$. Because of the injectivity we can write (3) (with the $0$ terms removed) as
\begin{equation}\label{eq4}
(S'|T')(S''|T'')=\sum\pm(S'_i|T'_i)(S''_i|T''_i)
\end{equation}
where $(S'_i,T'_i)\leq(S',T')$ and $(S'_i,T'_i)\leq(S''_i,T''_i)$.

If $I'$ and $J'$ are good then $I'\leq I''$ and $J'\leq J''$ which implies that
$S'\leq S''$ and $T'\leq T''$ contrary to the hypothesis.
Therefore one of $I'$ and $J'$,
say $I'$, must be bad. Since $I'$ is bad, $I'_i$ is good, and
$I'_i\leq I'$, we have $I'_i<I'$.
By Lemma~\ref{lem3} it follows that $S'_i < S'$
showing that $(S'_i,T'_i)<(S',T')$. The same argument applies if $J'$ is bad.
\end{proof}

\begin{remark}\label{rem1}
Suppose $A = A'\sqcup A''$ where $A'$ and $A''$ are disjoint
and $\varphi:A \to B$ is onto. Suppose $\varphi$ is
injective on $A'$ and on $A''$. Let $B' = \varphi(A')$ and $B'' = \varphi(A'')$.
Then $B'\cap B'' = \{x|\,|\varphi^{-1}(x)| = 2\}$ is independent of $A'$ and $A''$.
We can think of $B = \varphi(A)$ as a set with multiplicities where the multiplicity of
a point $x$ is the order of $\varphi^{-1}(x)$. If we think of $B'$ and $B''$ as
sets with all points of multiplicity $1$ then $B'\cup B'' = B$ as sets
with multiplicities. Applying this remark to the maps $\varphi:I\rightarrow S'\cup S''$
and $\psi:J\rightarrow T'\cup T''$ defined in the proof of Theorem~\ref{th3}
we see that the sets given by our proof of Theorem~\ref{th3} satisfy
$S'_i\cup S''_i = S'\cup S''$ and $T'_i\cup T''_i = T'\cup T''$ as multisets.
\end{remark}

\section{Standard monomials}
We say that a product $(A_1|B_1)\cdots(A_r|B_r)$ of the minors of
a matrix $X$ is a standard  monomial if $A_1\leq A_2\leq\cdots\leq A_r$ and
$B_1\leq B_2\leq\cdots\leq B_r$. We regard two standard monomials
which only differ by factors of the form $(\emptyset|\emptyset)$
as identical.
The following is an easy consequence of Theorem~\ref{th3}.

\begin{corollary}
Any polynomial in the entries of $X$ is a linear combination
of standard monomials in the minors of $X$.
\end{corollary}

\begin{proof}
Since $x_{ij}=(\{i\}|\{j\})$, it is clear that any such polynomial is a
linear combination of products of the minors of $X$.
We show that any product $(A_1|B_1)\cdots(A_r|B_r)$ with $r$ factors is a linear
combination of standard monomials with $r$ factors by induction on $r$ and on
$(A_1,B_1)$ in the finite partially ordered set of  pairs of subsets
of $\{1,\ldots,m\}$ and $\{1,\ldots,n\}$. By induction on $r$ we can
assume that $(A_2,B_2)\leq\cdots\leq(A_r,B_r)$. If
$(A_1,B_1)\leq(A_2,B_2)$ or $r=1$ we are done. If not, Theorem~\ref{th3} shows
that $(A_1|B_1)(A_2|B_2)=\sum\pm(C_i|D_i)(P_i|Q_i)$
where $(C_i,D_i)<(A_1,B_1)$ so we are done by induction on $(A_1,B_1)$.
\end{proof}

\begin{remark}
It follows from Remark~\ref{rem1} that if we write $(A_1|B_1)\cdots(A_r|B_r)$
as a linear combination of standard monomials
$(A^{(i)}_1|B^{(i)}_1)\cdots(A^{(i)}_{r_i}|B^{(i)}_{r_i})$ then
\newline $\bigcup_j A^{(i)}_j=\bigcup_j A_j$ and
$\bigcup_j B^{(i)}_j=\bigcup_j B_j$ for all $i$, counting
multiplicities. In other words the two sides have the same content in
the sense of \cite{DEPy}.
\end{remark}

To conclude, we give a proof of the following theorem which is rather
similar to the proof in \cite{DKR} but which uses no combinatorial
constructions. By a generic matrix we mean one whose entries are
distinct indeterminates.

\begin{theorem}\label{th4}
If $X$ is a generic matrix, the standard monomials in the minors of $X$
are linearly independent.
\end{theorem}

Before giving the proof we review some results about ordering monomials.
Given a totally ordered set of indeterminates, we can order the monomials in
these indeterminates as follows. If $x$ is an indeterminate and $m$ is such
a monomial write $\ord_x m$ for the number of times $x$ occurs in $m$.
If $m_1$ and $m_2$ are monomials define $m_1 > m_2$ to mean
$\ord _x m_1 > \ord _x m_2$ for some $x$ while $\ord _y m_1 = \ord _y m_2$
for all $y > x$. It is easy to check that this defines a total ordering on the set of monomials.

\begin{lemma}\label{lem4}If  $u_1, u_2, \cdots ,u_k$ and $v_1, v_2, \cdots ,v_k$ are
monomials  with $u_i \leq v_i$ for all $i$ and $u_i < v_i$ for some $i$ then
$u_1u_2 \cdots u_k < v_1v_2 \cdots v_k$.
\end{lemma}

It is sufficient to show $a < b$ implies $ac < bc$
and then replace the $u_i$'s by the $v_i$'s one by one.

It follows that if $f$ and $g$ are linear combinations of
monomials, the leading monomial of $fg$ is the product of the leading monomials of 
$f$ and $g$.

\begin{proof}[Proof of Theorem~\ref{th4}]
We specialize $X$ to a matrix of the form $X=YZ$ where $Y$ is a generic
$m\times N$ matrix, $Z$ is a generic $N\times n$ matrix and $N$ is
sufficiently large. By the classical Binet--Cauchy theorem we have
\begin{equation}\label{eq5}
X(A|B)=\sum_{S}Y(A|S)Z(S|B).
\end{equation}
This just expresses the functoriality of the exterior product:
\begin{equation}
\bigwedge(X)=\bigwedge(Y)\bigwedge(Z).
\end{equation}
By omitting $0$ terms in equation (\ref{eq5}) we can assume that $|A| = |B| = |S| = p$.
In the situation of equation (\ref{eq5}) let
$Y = (y^{(\nu)}_i)_{1\leq i \leq m,\, 1 \leq \nu \leq N}$ and
$Z = (z^{(\nu)}_j)_{1 \leq \nu \leq N\, 1 \leq j \leq n}$ be generic matrices
where the indeterminates $y^{(\nu)}_i$ and $z^{(\nu)}_j$ are all distinct.
Then $X = YZ$ has entries 
$x_{ij}=\sum_{\nu=1}^{N}y_i^{(\nu)}z_j^{(\nu)}$.
We order the indeterminates as follows:
\begin{equation}
y_1^{(1)}>\cdots>y_m^{(1)}>z_1^{(1)}>\cdots>z_n^{(1)}>y_1^{(2)}%
>\cdots>y_m^{(2)}>\cdots
\end{equation}
and order the monomials in these indeterminates as described above.

If $A=\{a_1<\cdots<a_p\}$ and $S=\{s_1<\cdots<s_p\}$, then
\begin{equation}
Y(A|S)=\sum_{\sigma\in\mathcal{S}_p}\pm y_{a_{\sigma 1}}^{(s_1)}\cdots y_{a_{\sigma p}}^{(s_p)}.
\end{equation}

The leading monomial of $Y(A|S)$ is $y_{a_1}^{(s_1)}\cdots y_{a_p}^{(s_p)}$
since $y_{a_1}^{(s_1)}\cdots y_{a_p}^{(s_p)} > 
y_{a_{\sigma1}}^{(s_1)}\cdots y_{a_{\sigma p}}^{(s_p)}$ for $\sigma \neq 1$.
To see this, let $i$ be least such that $\sigma i\neq i$. Then $\sigma i > i$ so
$y_{a_{\sigma i}}^{(s_i)} <  y_{a_i}^{(s_i)}$. Choose $x = y_{a_ i}^{(s_i)}$.
Then $\ord_xy_{a_1}^{(s_1)}\cdots y_{a_p}^{(s_p)} > \ord_x
y_{a_{\sigma1}}^{(s_1)}\cdots y_{a_{\sigma p}}^{(s_p)}$ while
$\ord_zy_{a_1}^{(s_1)}\cdots y_{a_p}^{(s_p)} = \ord_z
y_{a_{\sigma1}}^{(s_1)}\cdots y_{a_{\sigma p}}^{(s_p)}$ for $z > x$.

Similarly, if $B=\{b_1<\cdots<b_p\}$, then the leading monomial of $Z(S|B)$ is
$z_{b_1}^{(s_1)}\cdots z_{b_p}^{(s_p)}$. The various terms on the right hand side
of (\ref{eq5}) have leading monomials
$y_{a_1}^{(s_1)}\cdots y_{a_p}^{(s_p)}z_{b_1}^{(s_1)}\cdots z_{b_p}^{(s_p)}$.
Of these, the one with $s_i=i$ for all $i$ is the largest by
Lemma~\ref{lem4}. Therefore,
if $N \geq |A| = |B|$, the leading monomial of $X(A|B)$ is $y(A)z(B)$ where
$y(A)=y_{a_1}^{(1)}\cdots y_{a_p}^{(p)}$ and
$z(B)=z_{b_1}^{(1)}\cdots z_{b_p}^{(p)}$. It follows that the leading
term of $(A_1|B_1)\cdots(A_r|B_r)$ is $y(A_1)\cdots y(A_r)z(B_1)\cdots z(B_r)$.

Now if $A_1\leq A_2\leq\cdots\leq A_r$, and if $N\geq|A_1|$, we can recover
$A_1$ from $M=y(A_1)\cdots y(A_r)$ as follows.
Let $A_i=\{a_{i1}<\cdots<a_{ip_i}\}$. Then

\begin{equation}
M=\prod_jy_{a_{j1}}^{(1)} \cdots \prod_jy_{a_{js}}^{(s)}\cdots
\end{equation}

Since $A_1 \leq A_2 \leq \cdots \leq A_r$ we have
$a_{1s}\leq a_{2s}\leq\cdots$, and we see that $a_{1s}$ is the least $c$ such that
$y_c^{(s)}$ occurs in $M$. Note that $p_1 \geq p_2 \geq \cdots$ so that
$a_{1s}$ will exist if $a_{is}$ does.

Since $M/y(A_1)=y(A_2)\cdots y(A_r)$ we see that $M$ determines $A_2$, $A_3$, etc. if $N\geq|A_1|$. Similarly $z(B_1)\cdots z(B_r)$ determines $B_1$, $B_2$, etc.
Therefore the leading monomials of the
standard monomials $(A_1|B_1)\cdots(A_r|B_r)$
with $N\geq|A_1|$ and $N\geq|B_1|$ are all distinct and the theorem follows
since $N$ can be arbitrarily large.
\end{proof}

\begin{corollary}
For a generic square matrix $X$, all linear relations between the Laplace
products of $X$ are consequences of those in Theorem~\ref{th1}.
\end{corollary}

\begin{proof}
By Theorem~\ref{th2} the space of all Laplace products of $X$ is spanned by
those of the form $\{A|B\}$ with $A$ and $B$ good. For these
$\{A|B\}=\pm(A|B)(\widetilde A|\widetilde B)$ is a standard monomial
so these $\{A|B\}$ are linearly independent. Since the only relations needed
to prove Theorem~\ref{th2} are those of Theorem~\ref{th1}, the result follows.
\end{proof}

\end{document}